\documentclass[11pt,namelimits,sumlimits,a4paper]{amsart}
\usepackage{cite}

\usepackage{amssymb,amsmath}
\usepackage[mathscr]{eucal}
\usepackage{slashed}

\usepackage[latin1]{inputenc}
\usepackage[T1]{fontenc}
\usepackage[UKenglish]{babel}
\usepackage{amsfonts}

\usepackage{amsthm}
\usepackage{amsmath,amscd}
\usepackage{latexsym}

\usepackage{mathtools}

\usepackage{hyperref}

\usepackage{float}

\usepackage{esint}

\usepackage{layout}
\usepackage{enumitem}

\usepackage{lmodern}

\numberwithin{equation}{section}
\newtheorem{theorem}[equation]{Theorem}
\newtheorem{lemma}[equation]{Lemma}
\newtheorem{prop}[equation]{Proposition}

\theoremstyle{definition}
\newtheorem{definition}[equation]{Definition}

\theoremstyle{remark}
\newtheorem{remark}[equation]{Remark}
\newtheorem*{remark*}{Remark}

\newtheorem*{question*}{Question}

\newtheorem{conjecture}[equation]{Conjecture}

\newcounter{mtheorem}

\newtheoremstyle{mystyle}
  {}
  {}
  {\itshape}
  {}
  {\bfseries}
  {.}
  { }
  {}

\theoremstyle{mystyle}
\newtheorem{mtheorem}[mtheorem]{Theorem}

\newcommand{\ie}{\emph{i.e.} }
\newcommand{\eg}{\emph{e.g.} }
\newcommand{\cf}{\emph{cf.} }

\newcommand{\beq}{\begin{equation}}
\newcommand{\eeq}{\end{equation}}
\newcommand{\bea}{\begin{eqnarray}}
\newcommand{\eea}{\end{eqnarray}}

\newcommand{\C}{\mathbb{C}}

\newcommand{\R}{\mathbb{R}}
\newcommand{\Q}{\mathbb{Q}}
\newcommand{\Z}{\mathbb{Z}}

\newcommand{\HH}{\mathbb{H}}

\newcommand{\PP}{\mathbb{P}}

\newcommand{\ra}{\rightarrow}

\newcommand{\diag}{\operatorname{diag}}

\newcommand{\Ad}{\textrm{Ad}}
\newcommand{\ad}{\textrm{ad}}
\newcommand{\Lie}[1]{\mathfrak{#1}}

\newcommand{\tu}[1]{\textup{#1}}

\newcommand{\gtwo}{\ensuremath{\textup{G}_2}}

\newcommand{\unitary}[1]{\textup{U$(#1)$}}

\newcommand{\sunitary}[1]{\textup{SU$(#1)$}}

\newcommand{\sorth}[1]{\textup{SO$(#1)$}}

\newcommand{\triple}[1]{\underline{\boldsymbol{#1}}}
\newcommand{{\isomgtc}}{\ensuremath{\sunitary{2}^3 \rtimes S_3}}

\def\co{\colon\thinspace}

\begin{document}
\title{Deformations of hyperk\"ahler cones}
\author{Roger Bielawski
\and
Lorenzo Foscolo}

\address{Institut f\"ur Differentialgeometrie, Leibniz Universit\"at Hannover, Welfengarten 1, 30167 Hannover, Germany}
\address{Department of Mathematics, University College London, London WC1E 6BT, United Kingdom}

\begin{abstract}
We use twistor methods to promote Namikawa's universal Poisson deformations of conic affine symplectic singularities to families of hyperk\"ahler structures deforming hyperk\"ahler cone metrics. The metrics we produce are generally incomplete, but for specific classes of hyperk\"ahler cones, even these incomplete metrics have some interest and applications: we study in detail the case of the nilpotent cone of a simple complex Lie algebra, with applications to the definition of hyperk\"ahler quotients at arbitrary level of the moment map, and the case of Kleinian singularities.
\end{abstract}

\maketitle

\section{Introduction}


Construction and classification of hyperk\"ahler metrics have made great progress since Calabi \cite{Calabi} introduced the term ``hyperk\"ahler'' and gave the first nontrivial example (discovered almost simultaneously by Eguchi and Hanson \cite{Eguchi-Hanson}). Such metrics found numerous applications in various branches of mathematics and in theoretical physics. Their geometry is incredibly rich,  and can be studied using analytic, differential-geometric, or algebro-geometric methods. Not surprisingly, this also means that hyperk\"ahler metrics are very rigid. There are several instances where one can classify them without even  assuming the completeness.
\par
An instance of such a situation is one of the topics of the present paper:  $G$--invariant hyperk\"ahler metrics on open subsets of regular adjoint $G^\C$--orbits ($G$ is a compact  Lie group). In the special case where $G=\sunitary{2}$, these metrics are cohomogeneity one and explicit  \cite{Belinskii}. In the higher rank case, families of such metrics have been constructed by Kronheimer \cite{Kronheimer:Coadjoint,Kronheimer:Nilpotent}, and extended by Biquard and Kovalev \cite{Biquard:Nahm,Kovalev:Nahm}. These (and only these) metrics can be completed on the closure of the orbits. The existence of more general (incomplete) hyperk\"ahler metrics on open subsets of adjoint orbits has been conjectured  by Joyce \cite[\S 12]{Joyce:Hypercomplex}. D'Amorim Santa-Cruz \cite{D'Amorim:SantaCruz} has constructed these as {\em pseudo-hyperk\"ahler} manifolds, but was unable to prove that the metrics are positive-definite. As we discuss in \S\ref{arbitary;level},  knowing that these metrics are positive definite is of more than theoretical interest: it allows one to consider hyperk\"ahler quotients {\em at an arbitrary level}, similarly to symplectic quotients  $\mu^{-1}(l)/\tu{Stab}_G(l)$ at any level $l\in \Lie{g}^\ast$.
\par
D'Amorim Santa-Cruz constructed families of twistor spaces $Z_s$ parametrised by real sections $s$ of a certain holomorphic vector bundle over $\PP^1$.
As usual in such constructions, the hypercomplex manifold $M(s)$ of twistor lines in $Z_s$ may have several connected components, with different signatures of the metric.  We show here that this is indeed the case:
\begin{mtheorem}\label{thm:Santa:Cruz}
$M(s)$ always has a positive definite component and a negative definite one. However, for $G=\sunitary{3}$ and generic $s$, $M(s)$ also has a component with indefinite signature.  
\end{mtheorem}
We expect that the second statement holds for all simple Lie groups, apart from $\sunitary{2}$. We also prove another, rather unexpected property of $M(s)$: it has a positive-definite component $M(s)^+$, an end of which is asymptotic to the nilpotent cone in $\mathfrak{g}^\C$ with its (singular) Kronheimer's hyperk\"ahler metric. Thus in a sense $M(s)^+$ is only incomplete ``inside'', similarly to the $4$-dimensional $\sunitary{2}$--invariant hyperk\"ahler metrics of Belinsk\u{\i} et al. \cite{Belinskii} mentioned above.
\par
The proof of this result is in fact an application of a much more general twistor construction of hyperk\"ahler deformations of hyperk\"ahler cone metrics. Owing to a result of Donaldson and Sun \cite{Donaldson:Sun}, any {\em complete} hyperk\"ahler manifold $(M,g)$ with Euclidean volume growth has a unique tangent cone at infinity which is a normal affine variety $X$ with a $\C^\ast$--action with positive weights and carrying a hyperk\"ahler cone metric on its regular part. Thus on open subsets of its end, $(M,g)$ can be regarded as a hyperk\"ahler deformation of the cone metric on the regular part of $X$. Since such a cone $X$ can have non-isolated singularities, the construction of complete hyperk\"ahler metrics with $X$ as tangent cone at infinity via analytic methods is challenging. In fact, apart from Joyce's work \cite[Chapter 9]{Joyce:QALE} on QALE hyperk\"ahler metrics on crepant resolutions of quotient singularities $\HH^n/\Gamma$ for a finite group $\Gamma\subset \tu{Sp}(n)$, no general analytic construction is available and all other known examples arise as hyperk\"ahler quotients in finite (\eg hypertoric varieties \cite{Bielawski:Dancer}) or infinite (\eg moduli spaces of solutions to Nahm's equations \cite{Kronheimer:Nilpotent,Biquard:Nahm,Kovalev:Nahm}) dimensions.
\par
Here we use twistor methods and deformation theory to construct large families of hyperk\"ahler deformations of a hyperk\"ahler  cone. The main tool is Namikawa's universal graded Poisson deformation $\pi\co\Lie{X}\ra T_X$ of the conic affine symplectic variety $X$ \cite{Namikawa:Poisson:defs:affine}. We adapt the ideas of Hitchin \cite{Hitchin:Polygons:Gravitons}  and of D'Amorim Santa-Cruz \cite{D'Amorim:SantaCruz} to construct twistor spaces from $\pi\co\Lie{X}\ra T_X$, provided that $X$ is a hyperk\"ahler cone, and use deformation theory to also construct families of twistor lines:
\begin{mtheorem}\label{thm:QAC}
For every conic affine symplectic variety $X$ carrying a compatible hyperk\"ahler cone metric on its smooth part $X_{reg}$, there exists a family $\{(M_X(s)^+,g_s)\}$ of hyperk\"ahler metrics parametrised by real holomorphic sections of a certain vector bundle $T_X(2)\ra\PP^1$ and asymptotic to the cone metric on $X_{reg}$.
\end{mtheorem}
Generic $\{(M_X(s)^+,g_s)\}$ will be incomplete. If $X$ admits a symplectic resolution $\widetilde{X}$ (equivalently, if the generic fibre of $\pi\co \Lie{X}\ra T_X$ is smooth), the space of real sections of $T_X(2)\ra \PP^1$ contains a distinguished subset identified with a finite quotient of $H^2(\widetilde{X};\R)\otimes\R^3$. We conjecture that if $s$ is a generic element of this subset, then the completion of $(M_X(s)^+,g_s)$ yields a complete hyperk\"ahler metric on $\widetilde{X}$. Moreover, it is natural to wonder whether all complete hyperk\"ahler metrics with tangent cone at infinity $X$ must arise in this way.
  
The proof of Theorem \ref{thm:Santa:Cruz} is an application of Theorem \ref{thm:QAC} to the case where $X$ is the nilpotent cone of the complex Lie algebra $\Lie{g}^\C$. We also consider the application of Theorem \ref{thm:QAC} to the case where $X$ is a  Kleinian singularity $X=\C^2/\Gamma$ for some $\Gamma\subset \sunitary{2}$. We suggest that the incomplete ALE metrics on exterior domains of $X$ we produce in addition to Kronheimer's ALE metrics on the minimal resolution of $X$ \cite{Kronheimer:ALE:Construction,Kronheimer:ALE:Classification} play a role in the conjectural description by Pantev--Wijnholt \cite{Pantev:Wijnholt} of the local deformations of $\gtwo$--holonomy metrics with codimension-$4$ singularities in terms of $3$-dimensional gauge theory. 

\subsection*{Acknowledgments} RB is a member of the DFG Priority Programme 2026 ``Geometry at infinity''. LF is supported by a Royal Society University Research Fellowship. The authors would like to thank Travis Schedler for suggesting the proof of Lemma \ref{lem:Regular:Points:Central:Fibre}. The authors learned about \cite{Diaconescu:Donagi:Pantev}, which inspired Section \ref{sec:3dHitchin}, at the workshop ``Special Holonomy and Branes'' held at AIM in Autumn 2020: LF thanks AIM and the workshop organisers for the productive week and many of the workshop participants for discussions related to Section \ref{sec:3dHitchin}.

\section{Twistor construction of deformations of hyperk\"ahler cones}

In this section we use twistor methods to promote Namikawa's theory of universal Poisson deformations of affine symplectic singularities to a general construction of hyperk\"ahler metrics that can be regarded as deformations of hyperk\"ahler cones.

\subsection{Universal Poisson deformations of hyperk\"ahler cones}

We introduce the working definition of hyperk\"ahler cones we use in the paper, then review and refine Namikawa's construction of universal Poisson deformations of affine symplectic varieties with a good $\C^\ast$--action under the assumption that the regular part of the affine symplectic variety arises from a conical hyperk\"ahler structure.  

\subsubsection{Hyperk\"ahler cones}
For a thorough overview of 3-Sasakian geometry we refer the reader to \cite[Chapter 13]{Boyer:Galicki:Book}; here we recall only the basic properties we need.

A $3$-Sasaki manifold of dimension $4n-1\geq 3$ is a triple $(\Sigma,g_\Sigma,\triple{\eta})$, where $(\Sigma,g_\Sigma)$ is a smooth connected $(4n-1)$-dimensional Riemannian manifold and $\triple{\eta}=(\eta_1,\eta_2,\eta_3)$ is a hypercontact structure on $\Sigma$ satisfying the following condition: $\tu{C}(\Sigma)=\R_+ \times\Sigma$ endowed with the triple of closed $2$-forms $\triple{\omega}_{\tu{C}}=(\omega_1,\omega_2,\omega_3)$, $\omega_i = d\left( \tfrac{1}{2}r^2 \eta_i\right)$ for $i=1,2,3$, is a hyperk\"ahler manifold with induced cone metric $g_\tu{C}=dr^2 + r^2 g_\Sigma$. Since hyperk\"ahler metrics are Ricci-flat, $g_\Sigma$ is necessarily an Einstein metric with positive scalar curvature $(4n-1)(4n-2)$. Hence complete $3$-Sasaki manifolds are compact and have finite fundamental group.

The tangent bundle of $\Sigma$ has a $g_\Sigma$--orthogonal decomposition $T\Sigma = \R \xi_1 \oplus \R\xi_2 \oplus \R \xi_3 \oplus \mathcal{H}$, where $\xi_i$ is the dual vector field to $\eta_i$ and $\mathcal{H}=\ker{(\eta_1,\eta_2,\eta_3)}$. The vector fields $\xi_1,\xi_2,\xi_3$ generate an infinitesimal isometric action of $\Lie{su}(2)$ that acts as infinitesimal rotations of the triple $\triple{\eta}$. If $\Sigma$ is complete this infinitesimal action integrates to an action of $\sunitary{2}$ with the same properties. The group acting effectively is either $\sunitary{2}$ or $\sorth{3}$ \cite[Proposition 13.3.11]{Boyer:Galicki:Book}. In order to fix ideas, in the rest of the section we will assume the latter case occurs. For example, if $\Sigma$ is the Konishi bundle over a smooth quaternionic K\"ahler manifold then the group acting effectively is always $\sorth{3}$ unless $(\Sigma,g_\Sigma)$ is the round $(4n-1)$-sphere \cite[Theorem 6.3]{Salamon:qK}.

Every compact  $3$-dimensional $3$-Sasaki manifold is of the form $S^3/\Gamma$ for some finite subgroup $\Gamma$ of $\sunitary{2}$. In higher dimensions infinitely many compact $3$-Sasaki manifolds are known \cite{Boyer:Galicki:3Sasaki:7d}, but there also many interesting incomplete examples giving rise to hyperk\"ahler cones with non-isolated singularities. In this note we use complex geometry to provide a unified framework to work with such ``singular'' hyperk\"ahler cones using twistor methods. As a justification for our approach, the next proposition describes the underlying complex geometric structure of a hyperk\"ahler cone $\tu{C}(\Sigma)$ over a compact $3$-Sasaki manifold $\Sigma$.     

\begin{prop}\label{prop:Smooth:HK:Cones}
Let $\tu{C}(\Sigma)$ be the hyperk\"ahler cone over a compact $3$-Sasaki manifold $\Sigma$. Fix a complex structure compatible with the hyperk\"ahler metric and denote by $\omega^c$ the corresponding holomorphic symplectic form. Then the complex structure on $\tu{C}(\Sigma)$ induces on $X=\tu{C}(\Sigma)\cup \{ o\}$, where $o$ denotes the vertex of the cone, the structure of a complex analytic space such that $(X,\omega^c)$ is an affine symplectic variety with a $\C^\ast$--action acting with positive weights on the coordinate ring of $X$ and with weight $l=1$ on $\omega^c$ and fixing only $o\in X$. Moreover, $X$ admits an antiholomorphic involution $j\co X\ra X$ such that $j^\ast\omega^c = \overline{\omega^c}$.
\proof
Denote by $J_i$, $i=1,2,3$, the complex structure on $\tu{C}(\Sigma)$ corresponding to the K\"ahler form $\omega_i$. Up to a hyperk\"ahler rotation we can assume that the distinguished complex structure in the statement is $J_1$. The corresponding holomorphic symplectic form is $\omega^c = \omega_2 + i\omega_3$.

The vector field $\xi_1$ on $\Sigma$ extended radially to $\tu{C}(\Sigma)$ satisfies $\mathcal{L}_{\xi_1}\omega^c = 2i \omega^c$ and $J_1\xi_1 = - r\partial_r$. Since we assume that only $\sorth{3}$ acts effectively, $\tfrac{1}{2}\xi_1$ has period $2\pi$ and generates a $\C^\ast$--action on $\tu{C}(\Sigma)$ without fixed points acting with weight $l=1$ on $\omega^c$ as claimed. Identifying the double cover $\sunitary{2}$ of $\sorth{3}$ with the sphere of unit quaternions, the action of $j$ on $\tu{C}(\Sigma)$ yields an anti-holomorphic involution, still denoted by $j$, satisfying $j^\ast\omega^c =\overline{\omega^c}$.

The top power of the holomorphic symplectic form defines a holomorphic complex volume form on the K\"ahler cone $(\tu{C}(\Sigma),J_1,\omega_1)$. By \cite[Theorem 3.1 and Proposition 3.5]{vanCoevering:Examples} $X$ is then a normal affine variety and $o$ is a rational Gorenstein singularity. The $\C^\ast$--action on $\tu{C}(\Sigma)$ and real structure $j$ extend to $X$ fixing the singular point $o$. Given the existence of the holomorphic symplectic form $\omega^c$ on the smooth part $\tu{C}(\Sigma)$ of $X$, by \cite[Theorem 6]{Namikawa:Extension:2forms} $X$ is an affine symplectic variety in the sense of \cite{Beauville:Symplectic:Singularities}, \ie $\omega^c$ extends to a (not necessarily non-degenerate) holomorphic $2$-form on any resolution $p\co X'\ra X$.
\endproof  	
\end{prop}

\begin{remark*}
If $\sunitary{2}$ instead of $\sorth{3}$ acts effectively on $\Sigma$, then we would have $l=2$ instead of $l=1$ in the statement of the Proposition. 	
\end{remark*}

\begin{remark}\label{rmk:Good:Cast:action}
A $\C^\ast$--action on a pointed affine symplectic variety $(X,o,\omega^c)$ with the properties in the statement of the Proposition, \ie satisfying \cite[Condition $({}^\ast)$]{Namikawa:Poisson:defs:affine}, will be referred to as a good $\C^\ast$--action on $(X,o,\omega^c)$.
\end{remark}

We now turn the conclusion of Proposition \ref{prop:Smooth:HK:Cones} into a working definition of a possibly singular hyperk\"ahler cone as a global object. 

\begin{definition}\label{def:Singular:HK:Cones}
A \emph{conical hyperk\"ahler variety} $(X,o,\omega^c,t,j)$ is a normal affine variety $X$ with a distinguished point $o\in X$ endowed with
\begin{enumerate}
\item a holomorphic symplectic form $\omega^c$ on the smooth part $X_{reg}$ of $X$ which makes $(X,\omega^c)$ into an affine symplectic variety in the sense of \cite{Beauville:Symplectic:Singularities};
\item a good $\C^\ast$--action $\C^\ast \stackrel{t}{\ra}\tu{Aut}(X)$ on $(X,o,\omega^c)$ in the sense of Remark \ref{rmk:Good:Cast:action};
\item an antiholomorphic involution $j\co X\ra X$ such that $j^\ast\omega^c=\overline{\omega^c}$ and $j\circ t = \overline{t}\circ j$ for all $t\in \C^\ast$ (here by abuse of notation for every $t\in \C^\ast$ we denote by the same symbol $t$ the corresponding automorphism of $X$);
\item the smooth locus $X_{reg}$ of $X$ carries a conical hyperk\"ahler structure $(g_{\tu{C}},\omega_1,\omega_2,\omega_3)$ such that $\omega^c = \omega_2 + i\omega_3$, the complex structure $J_1$ corresponding to $\omega_1$ is induced from the complex structure on $X_{reg}$ and the $\C^\ast$--action $t$ is generated by the Reeb vector field $\tfrac{1}{2}\xi_1$.  	
\end{enumerate}
\end{definition}

\begin{remark*}
In the variant of the definition where the weight of the $\C^\ast$--action on the complex symplectic form is $l=2$ instead of $l=1$, the $\C^\ast$--action is generated by the vector field $\xi_1$ on $X_{reg}$. 
\end{remark*}

\begin{remark}\label{rmk:Singular:HK:Cones}
In \cite{Donaldson:Sun} Donaldson--Sun have shown that non-collapsed tangent cones (at singular points as well as tangent cones at infinity in the non-compact case) of Gromov--Hausdorff limits of K\"ahler Einstein metrics are unique and have the structure of normal affine varieties. In the hyperk\"ahler case it seems likely that the additional assumptions of Definition \ref{def:Singular:HK:Cones} will be automatically satisfied.    
\end{remark}

\subsubsection{Universal Poisson deformations}

We want to study hyperk\"ahler deformations of hyperk\"ahler cones. Definition \ref{def:Singular:HK:Cones} allows us to make sense of these at least at the level of the underlying complex analytic data. Indeed, drop for a moment the real structure $j$ and parts (iii) and (iv) in Definition \ref{def:Singular:HK:Cones} and focus on the affine symplectic variety $(X,o,\omega^c,t)$ with a good $\C^\ast$--action. One can then introduce the notion of a graded Poisson deformation $\pi\co \Lie{X}\ra T$ of $(X,o,\omega^c,t)$: a holomorphic Poisson variety $\Lie{X}$ with a $\C^\ast$--action preserving the Poisson structure, a flat morphism $\pi\co \Lie{X}\ra T$ which is $\C^\ast$--equivariant with respect to a $\C^\ast$--action on the affine complex space $T$ with a unique fixed point $0\in T$, and an isomorphism between $\pi^{-1}(0)$ and $X$ endowed with its $\C^\ast$--action and the Poisson structure induced by the holomorphic symplectic form $\omega^c$. 

In \cite{Namikawa:Poisson:defs:affine} Namikawa constructs a \emph{universal} graded Poisson deformation of $(X,o,\omega^c,t)$, \ie a graded flat Poisson deformation $\pi\co \Lie{X}\ra T_X$ such that any other such deformation $\pi'\co \Lie{X}'\ra T'$ of $(X,o,\omega^c,t)$ is induced by a unique morphism $T'\ra T_X$.    

\begin{theorem}[Namikawa, \cite{Namikawa:Poisson:defs:affine}]\label{thm:Namikawa:Poisson:def}
Let $(X,o,\omega^c,t)$ be an affine symplectic variety with a good $\C^\ast$--action. Then there exists a universal graded Poisson deformation $\pi\co \Lie{X}\ra T_X$ of $(X,o,\omega^c,t)$ with $T_X\simeq \C^r$.	
\end{theorem}

\begin{remark*}
Namikawa works in the algebraic category (or even with formal schemes when there is no $\C^\ast$--action): here we work with the induced complex analytic spaces.	
\end{remark*}

We make the following immediate observations about Namikawa's universal Poisson deformation $\pi\co \Lie{X}\ra T_X$. Without loss of generality we can assume that the $\C^\ast$--action on the base $T_X\simeq \C^r$ is diagonal and denote by $d_1,\dots, d_r\in \Z$ its weights; since the $\C^\ast$--action on $T_X$ has positive weights \cite[\S 5.4]{Namikawa:Poisson:defs:affine}, we have $d_i \geq 1$ for all $i$. The weight of the $\C^\ast$--action on the Poisson structure of $\Lie{X}$ must be $-l=-1$ since this is the case for $X$. Finally, since the Poisson structure on $X$ is non-degenerate, the Poisson structure on every fibre of $\pi$ is also non-degenerate (this is true for nearby fibres by openness and is then extended to any fibre using the $\C^\ast$--action) and therefore it is induced by a family of fibre-wise holomorphic symplectic forms, which by abuse of notation we still denote by $\omega^c$.

For our applications it is crucial to bring back into the discussion the real structure $j$ in part (iii) of Definition \ref{def:Singular:HK:Cones}.
Since we do not know whether the universal deformation must be reduced, we define first real structures for general complex spaces. Locally, such a complex space $(Y,\mathcal{O}_Y)$ is isomorphic to a model complex space $(V,\mathcal{O}_V)$, where $V$ is the common zero set of finitely many holomorphic functions $f_1,\dots,f_k$ in a domain $D$ of $\C^n$, and $\mathcal{O}_V$ is the quotient of $\mathcal{O}_D$ by the ideal sheaf generated by these functions. Such a model space has a well-defined conjugate space  $(V,\overline{\mathcal{O}_V})$, where the sheaf $\overline{\mathcal{O}_V}$  is the quotient of $\overline{\mathcal{O}_D}$ (antiholomorphic functions) by the ideal sheaf generated by $\overline{f_1},\dots,\overline{f_k}$. If $\mathcal{O}_V$ has a Poisson structure, conjugation morphisms between $\mathcal{O}_V$ and $\overline{\mathcal{O}_V}$ allow to define a conjugate Poisson structure on the latter. A (Poisson) isomorphism between two local model spaces is also a (Poisson) isomorphism between their conjugates, and hence we obtain the conjugate (Poisson) complex space 
$(Y,\overline{\mathcal{O}_Y})$ by using the same gluing maps as for $(Y,\mathcal{O}_Y)$, but applied to the conjugate model spaces. 
Now a {\em real structure} $\sigma$ on a (Poisson) complex space $(Y, \mathcal{O}_Y)$ is a (Poisson) involution $(\sigma,\tilde\sigma):(Y, \mathcal{O}_Y)\to (Y, \overline{\mathcal{O}_Y})$.

\begin{lemma}\label{lem:Namikawa:real}
Let $(X,o,\omega^c,t,j)$ be a conical hyperk\"ahler variety and denote by $\pi\co \Lie{X}\ra T_X$ the universal graded Poisson deformation of $(X,o,\omega^c,t)$ given in Theorem \ref{thm:Namikawa:Poisson:def}. Then $\Lie{X}$ and $T_X$ have real structures $\sigma$ and $\tau$ compatible with the $\C^\ast$--actions and such that $\pi\circ\sigma = \tau\circ\sigma$.
\proof
The proof follows \cite[Lemma 2.1]{Donaldson:Friedman} and uses the universality of Namikawa's Poisson deformation of $X$. For reasons of brevity we denote a conjugate complex space $(Y,\overline{\mathcal{O}_Y})$  by $\overline{Y}$. If $\pi\co \Lie{X}\ra T_X$ of $(X,o,\omega^c,t)$ is the universal graded Poisson deformation of $X$, then $\overline{\pi}\co \overline{\Lie{X}}\ra \overline{T_X}$ is a graded Poisson deformation of $(\overline{X},o,\overline{\omega^c},\overline{t})$, which, using the real structure $j$ of $X$, we can identify with a graded Poisson deformation of $(X,o,\omega^c,t)$ itself. Therefore, owing to the (uni)versality of $\pi\co \Lie{X}\ra T_X$, there exists a commutative diagram
\[
\begin{CD}
\overline{\Lie{X}} @>\sigma >> \Lie{X}\\
@V\overline{\pi}VV  @VV\pi V\\
\overline{T_X} @> \tau >> T_X
\end{CD}
\] 
compatible with the $\C^\ast$--actions and the Poisson structures such that the morphism $\sigma$ restricts to $j$ on the central fibre. Iterating $\sigma$ and using the universal property of $\pi\co \Lie{X}\ra T_X$ shows that $\sigma^2=\tu{id}=\tau^2$ and therefore $\sigma$ defines a real structure on $\Lie{X}$.    
\endproof
\end{lemma}

\begin{remark}\label{rmk:Namikawa:real}
As explained in \cite[Remark (a), \S 4.2]{Donaldson:Friedman} we can assume without loss of generality that the real structure $\tau$ on $T_X\simeq \C^r$ is the standard complex conjugation.	 Compatibility of the real structure with the $\C^\ast$--action implies that in these coordinates the $\C^\ast$--action on $T_X$ remains diagonal. 
\end{remark}

The following definition will play a key role in distinguishing a special set of twistor lines. 

\begin{definition}\label{def:Regular:Poisson:def}
A point $m\in \Lie{X}$ is called \emph{regular} if $m$ is a smooth point of $\Lie{X}$ at which $\pi\co\Lie{X}\ra T_X$ is a holomorphic submersion. We denote the set of regular points of $\Lie{X}$ by $\Lie{X}_{reg}$.	
\end{definition}

Note that for a point $m\in\Lie{X}_{reg}$ we have an exact sequence of vector spaces
\begin{equation}\label{eq:Regular:twistor:line:pointwise}
0\ra T^v_m\Lie{X}\longrightarrow T_m \Lie{X}\stackrel{d_m\pi}{\longrightarrow}\C^r\ra 0,
\end{equation}
where $T^v_m\Lie{X}$ is the tangent space to the symplectic leaf of $\Lie{X}$ containing $m$.

\begin{lemma}\label{lem:Regular:Points:Central:Fibre}
$\Lie{X}_{reg}\cap \pi^{-1}(0)= X_{reg}$.
\proof
It is clear from Definition \ref{def:Regular:Poisson:def} that $\Lie{X}_{reg}\cap \pi^{-1}(0)\subseteq X_{reg}$. The opposite inclusion follows from the fact that $\pi\co\Lie{X}\ra T_X$ is a locally trivial flat deformation at points of $X_{reg}$, \ie every point $p\in X_{reg}$ has a neighbourhood $\mathcal{U}_p\subset\Lie{X}$ with $\mathcal{U}_p\cap \pi^{-1}(0)\subset X_{reg}$ such that $\mathcal{U}_p\simeq (\mathcal{U}_p\cap \pi^{-1}(0))\times \pi (\mathcal{U}_p)$. This follows from smoothness of the deformation space $T_X$ in Theorem \ref{thm:Namikawa:Poisson:def} and the fact that Poisson deformations of a holomorphic symplectic disc are trivial by the existence of Darboux coordinates. 
\endproof
\end{lemma}

\subsection{The twistor construction}

We now promote the complex deformations studied in the previous subsection to hyperk\"ahler deformations via twistor methods.

\subsubsection{The twistor space}

Let $(X,o,\omega^c,t,j)$ be a conical hyperk\"ahler variety in the sense of Definition \ref{def:Singular:HK:Cones} and denote by $\pi\co\Lie{X}\ra T_X$ its universal graded Poisson deformation together with the real structures $\sigma$ and $\tau$ of Lemma \ref{lem:Namikawa:real}.

The idea now is to use the $\C^\ast$--action of $\Lie{X}$ to construct a fibration over $\PP^1$ with fibre $\Lie{X}$ by twisting the principal $\C^\ast$--bundle $\mathcal{O}(2)^\times$ (i.e.\ the degree $2$ line bundle with the $0$-section removed) by $\Lie{X}$.

\begin{remark*}
If the weight of the $\C^\ast$--action on the symplectic form $\omega^c$ is $l=2$, which arises from a hyperk\"ahler cone $\tu{C}(\Sigma)$ with an effective action of $\sunitary{2}$ instead of $\sorth{3}$, we twist instead $\mathcal{O}(1)^\times$.	
\end{remark*}

We will denote the total space of this fibration by $\Lie{X}(2)$. Explicitly, $\Lie{X}(2)$ is the union of two copies of $\C\times\Lie{X}$ identified over $\C^\ast\times\Lie{X}$ via $(\zeta',x') = (\zeta^{-1},\zeta^{-2}\cdot x)$. Define in a similar way $T_X(2)$ and observe that $T_X(2)\simeq \bigoplus_{i=1}^r{\mathcal{O}(2d_i)}$. The projection $\pi\co\Lie{X}\ra T_X$ then becomes a holomorphic fibre bundle map $\pi\co\Lie{X}(2)\ra T_X(2)$ over $\PP^1$.

The antipodal map and the real structure of $T_X$ induce a real structure on sections of $T_X(2)$.  

\begin{prop}\label{prop:Twistor:space}
Let $s$ be a real section of $T_X(2)$ and denote by $Z_s$ the smooth locus of $\pi^{-1}(s)\subset \Lie{X}(2)$. Then $Z_s$ is a hyperk\"ahler twistor space, \ie there is a holomorphic projection $p\co Z_s\ra \PP^1$, a holomorphic section $\omega^c$ of $\Lambda^2T_F^\ast Z_s\otimes\mathcal{O}(2)$ which restricts to a symplectic form on each fibre of $p$ and a compatible real structure lifting the antipodal map on $\PP^1$.
\proof
It is clear that $Z_s$ is a holomorphic fibre bundle $p\co Z_s\ra\PP^1$, the fibres of which are symplectic leaves of $\Lie{X}$. The (holomorphic) family of holomorphic symplectic forms $\omega^c$ on the symplectic leaves of $\Lie{X}$ induces the holomorphic section of $\Lambda^2T_F^\ast Z_s\otimes\mathcal{O}(2)$ which restricts to a symplectic form on each fibre. Here the factor $\mathcal{O}(2)$ arises precisely from the fact that the $\C^\ast$--action on $\Lie{X}$, which we use to twist $\mathcal{O}(2)^\times$ in the construction of $\Lie{X}(2)$, acts with weight $-1$ on its Poisson bracket (which is dual to $\omega^c$ on the symplectic leaves).

Finally, since $s$ is real, the real structure $\sigma$ on $\Lie{X}$ induces a real structure on $Z_s$ compatible with the fibre-wise holomorphic symplectic form and inducing the antipodal map on $\PP^1$. 
\endproof 
\end{prop}

\subsubsection{Twistor lines} 

In order to construct a hyperk\"ahler metric from $Z_s$ we must consider the existence of real holomorphic sections of $Z_s$ with normal bundle a sum of $\mathcal{O}(1)$'s. 
 
The key notion, inspired by \cite{D'Amorim:SantaCruz}, is the one of a \emph{regular twistor line}.

\begin{definition}
A section of $m$ of $\Lie{X}(2)\ra \PP^1$ is called \emph{regular} if $m(\zeta)\in \Lie{X}_{reg}$ for all $\zeta\in\PP^1$.	
\end{definition}
 
For a regular section $m$ of $\Lie{X}(2)$ with $\pi (m) = s\in H^0 (\PP^1;T_X(2))$ we have a sequence of holomorphic vector bundles over $\PP^1$
\begin{equation}\label{eq:Regular:twistor:line}
0\ra m^\ast T^v_F\Lie{X}(2) \longrightarrow m^\ast T_F\Lie{X}(2) \stackrel{d_m\pi}{\longrightarrow} T_X(2)\simeq\bigoplus_{i=1}^r{\mathcal{O}(2d_i)}\ra 0
\end{equation}
which over every point of $\PP^1$ reduces to \eqref{eq:Regular:twistor:line:pointwise}. The first bundle in \eqref{eq:Regular:twistor:line} is precisely the normal bundle of $m$ in $Z_s$.

\begin{lemma}\label{lem:Regular:twistor:line}
A regular real section $m$ of $Z_s$ is a twistor line, \ie it has normal bundle in $Z_s$ a sum of $\mathcal{O}(1)$'s, if and only if
\begin{equation}\label{eq:Strongly:regular:twistor:line}
d_m\pi\co H^0 (\PP^1; m^\ast T_F\Lie{X}(2) \otimes\mathcal{O}(-2))\ra H^0(\PP^1; \bigoplus_{i=1}^r{\mathcal{O}(2d_i-2)})
\end{equation}
is an isomorphism.
\proof
Note that $m$ is a twistor line if and only if $m^\ast T^v_F\Lie{X}(2)\otimes\mathcal{O}(-2)$ has vanishing cohomology. Tensoring \eqref{eq:Regular:twistor:line} by $\mathcal{O}(-2)$ and passing to the long exact sequence in cohomology shows that the kernel and cokernel of the map in \eqref{eq:Strongly:regular:twistor:line} are identified with, respectively, $H^0(\PP^1;m^\ast T^v_F\Lie{X}(2)\otimes\mathcal{O}(-2))$ and the kernel of $H^1(\PP^1;m^\ast T^v_F\Lie{X}(2)\otimes\mathcal{O}(-2))\ra H^1(\PP^1;m^\ast T_F\Lie{X}(2)\otimes\mathcal{O}(-2))$. Note however that the fibre-wise holomorphic symplectic form on $Z_s$ yields an isomorphism $m^\ast T_F^v\mathcal{X}(2) \otimes\mathcal{O}(-2) \simeq \left(m^\ast T_F^v\mathcal{X}(2)\right)^\ast$, which, combined with Serre duality, shows that $H^0(\PP^1;m^\ast T^v_F\Lie{X}(2)\otimes\mathcal{O}(-2))$ and $H^1(\PP^1;m^\ast T^v_F\Lie{X}(2)\otimes\mathcal{O}(-2))$ are dual vector spaces.     
\endproof	
\end{lemma}

The twistor correspondence \cite[Theorem 3.3]{HKLR} yields the existence of a pseudo-hyperk\"ahler structure $(g_s,\triple{\omega}_s)$ on each component of the space of twistor lines of $Z_s$ and by restriction on each non-empty component of the space $M_X(s)$ of \emph{regular} twistor lines. We want to show that there exists a non-empty component $M_X(s)^+$ of $M_X(s)$ with positive definite hyperk\"ahler structure.

We make two preliminary observations. Firstly, $Z_0$ is a fibration over $\PP^1$ with fibres the smooth locus of $X$. By part (iv) of Definition \ref{def:Singular:HK:Cones} there exists a component of the space of twistor lines of $Z_0$ corresponding to the conical hyperk\"ahler metric on $X_{reg}$. By Lemma \ref{lem:Regular:Points:Central:Fibre} any such twistor line is in fact regular. We denote by $M_X(0)^+$ the non-empty component of regular twistor lines of $Z_0$ corresponding to the hyperk\"ahler cone metric $g_0$ on $X_{reg}$. Evaluation of the section $m\in M_X(0)^+$ at any fixed point in $\PP^1$ identifies $M_X(0)^+$ with $X_{reg}$.

Our second observation is that for every non-zero $t\in\R^+$, $M_X(t\cdot s)$ is just $M_X(s)$ with scaled pseudo-hyperk\"ahler metric $g_{t\cdot s} = t\, g_s$. Indeed, assuming as in Remark \ref{rmk:Namikawa:real} that the real structure $\tau$ on $T_X\simeq \C^r$ is standard complex conjugation, we immediately have that $t\cdot s$ remains a real section whenever $s$ is. Secondly, $\C^\ast$--equivariance of Namikawa's universal Poisson deformation $\pi\co \Lie{X}\ra T_X$ yields a commutative diagram
 \[
\begin{CD}
T_m\Lie{X} @>d_m\pi >> T_s T_X \\
@V d_m tVV  @VV d_st V\\
T_{t\cdot m}\Lie{X} @> d_{t\cdot m}\pi >> T_{t\cdot s}T_X
\end{CD}
\] 
at any point $m\in \Lie{X}_{reg}$ with $s=\pi(m)$. The vertical maps are isomorphisms that induce identifications of the vector bundles in the exact sequences \eqref{eq:Regular:twistor:line} corresponding to the sections $m$ and $t\cdot m$. It then follows from the cohomological interpretation of Lemma \ref{lem:Regular:twistor:line} that $t\cdot m$ is a regular twistor line whenever $m$ is. Since the action of $t$ identifies the twistor spaces $Z_s$ and $Z_{t\cdot s}$ and $\C^\ast$ acts with weight $1$ on the fibrewise holomorphic symplectic form, the claim follows. Note in particular that, since the $\C^\ast$--action on $T_X$ fixes the origin, this observation recovers the fact that $g_0$ is scale invariant. 

\begin{theorem}\label{thm:HK:def:HK:cones}
Let $(X,o,\omega^c,t,j)$ be a conical hyperk\"ahler variety. Then for every real section $s$ of $T_X(2)\simeq \bigoplus_{i=1}^r{\mathcal{O}(2d_i)}$ there exists a component $M_X(s)^+$ of regular twistor lines of $Z_s$ such that the pseudo-hyperk\"ahler metric $g_s$ is positive definite on $M_X(s)^+$.
\proof
Fix a real section $s$ of $T_X(2)$ and a regular twistor line $m_0\in M_X(0)^+$. By the scaling property just discussed, it suffices to show that we can deform $m_0$ to a regular twistor line $m_t$ of $Z_{t\cdot s}$ for $t\in \R^+$ sufficiently small and that $g_{t\cdot s}$ is positive definite at $m_t$.
 
Now, since $m_0^\ast T^v_F\Lie{X}(2)$ is a sum of $\mathcal{O}(1)$'s and $d_i\geq 1$, the long exact sequence in cohomology induced by \eqref{eq:Regular:twistor:line} shows that
\begin{enumerate}
\item $H^1 (\PP^1; m_0^\ast T_F\Lie{X}(2))=0$, so the space of sections of $\Lie{X}(2)$ is smooth at $m_0$;
\item the map $d_{m_0}\pi\co H^0 (\PP^1; m_0^\ast T_F\Lie{X}(2))\ra H^0(\PP^1; T_X(2))$ is surjective.	
\end{enumerate}
An application of the Implicit Function Theorem implies that $m_0$ can be deformed to a section of $Z_{t\cdot s}$ for all $t$ sufficiently small. Up to taking $t$ even smaller if necessary, $m_t$ is automatically a regular twistor line of $Z_s$ and the metric $g_{t\cdot s}$ is automatically positive definite at $m_{t}$ by openness of these conditions.
\endproof 
\end{theorem}

The hyperk\"ahler manifold $(M_X(s)^+,g_s)$ is \emph{asymptotic} to the hyperk\"ahler cone $(X_{reg},g_0)$ in the following sense. Fix a relatively compact open set $\mathcal{U}$ of $M_X(0)^+\cap \{ r=1\}\subset X_{reg}$ and consider the set of rays $\tu{C}_1(\mathcal{U}):=\{ t\cdot\mathcal{U}\, |\, t\geq 1\}\subset M_X(0)^+$. The application of the Implicit Function Theorem in the proof of Theorem \ref{thm:HK:def:HK:cones} and the scaling action discussed before the Theorem imply that there exists an open neighbourhood $\mathcal{V}$ of $0\in H^0(\PP^1;T_X(2))^\R$ such that for every $s\in\mathcal{V}$ the set $M_X(s)^+$ contains an open set diffeomorphic (actually, isomorphic in the real analytic category) to $\tu{C}_1(\mathcal{U})$. Fix any non-zero $s\in\mathcal{V}$. Without loss of generality we can assume that $\varepsilon\cdot s\in \mathcal{V}$ for any $\varepsilon\in [0,1]$. Then the scaling property we have discussed implies that, as $\varepsilon\ra 0$, the family of blow-downs $(\tu{C}_1(\mathcal{U}), \varepsilon g_{s}) = (\tu{C}_1(\mathcal{U}), g_{\varepsilon\cdot s})$ converges smoothly (in fact, real analytically with respect to the real analytic structures induced by the holomorphic structures of the twistor spaces and their spaces of sections) to the hyperk\"ahler cone $(\tu{C}_1(\mathcal{U}), g_{0})$.   

\subsubsection{Symplectic resolutions}

As we will see in the next section the hyperk\"ahler metrics produced by Theorem \ref{thm:HK:def:HK:cones} are in general incomplete. Examples however suggest that complete maximal volume growth metrics with tangent cone at infinity $X$ arise from the following special situation.

Assume that the symplectic variety with a good $\C^\ast$--action $(X,o,\omega^c,t)$ admits a symplectic resolution $\widetilde{X}\ra X$, \ie a resolution of singularities with the further property that the pull-back of $\omega^c$ extends to a holomorphic symplectic form on $\widetilde{X}$ (in particular $\widetilde{X}\ra X$ is a crepant resolution). Namikawa \cite{Namikawa:Poisson:defs:affine} showed that $\widetilde{X}$ inherits a good $\C^\ast$--action and has a universal graded Poisson deformation $\widetilde{\pi}\co \widetilde{\Lie{X}}\ra \Lie{h}^\C$ fitting into a $\C^\ast$--equivariant diagram:
 \begin{equation}\label{eq:Symplectic:Resolution}
\begin{CD}
\widetilde{\Lie{X}} @>>> \Lie{X}\\
@V\widetilde{\pi}VV  @VV\pi V\\
\Lie{h}^\C @>q>> T_X
\end{CD}
\end{equation}
Here $\Lie{h}^\C\simeq H^2(\widetilde{X};\C)$ and there exists a finite Coxeter group $W$ acting linearly on $\Lie{h}^\C$ such that $T_X\simeq \Lie{h}^\C/W$ and $q$ is the quotient map. Denote by $R$ the collection of linear functionals on $\Lie{h}^\C$, the kernels of which are the reflection hyperplanes associated with the generators of the Coxeter group $W$. The $\C^\ast$--action on $\Lie{h}^\C$ is the standard scaling action with weights $(1,\dots,1)$ \cite[\S 5.4]{Namikawa:Poisson:defs:affine}. Furthermore, if $\Lie{X}'$ denotes the family $\Lie{X}'=\Lie{X}\times_{T_X}\Lie{h}^\C\stackrel{\pi'}{\ra} \Lie{h}^\C$ obtained by base change, then for every $\xi\in\Lie{h}^\C$ the fibre $\widetilde{\Lie{X}}_{\xi}$ is a smooth affine variety isomorphic to $\Lie{X}'_\xi$ whenever $\alpha (\xi)\neq 0$ for all $\alpha\in R$, and it is a smooth symplectic resolution of $\Lie{X}'_{\xi}$ otherwise.

\begin{remark*}
More generally one can work with $\widetilde{X}$ a $\Q$--factorial terminalisation of $X$. Namikawa shows that $\widetilde{\Lie{X}}$ is a locally trivial deformation of $\widetilde{X}$. In particular, $\widetilde{\Lie{X}}_\xi$, and therefore $\Lie{X}'_\xi$ for generic $\xi$, is smooth if and only if $\widetilde{X}$ is. Namikawa concludes \cite{Namikawa:Poisson:defs:affine} that $X$ admits a symplectic resolution if and only if the generic fibre of the universal Poisson deformation of $X$ is smooth.	
\end{remark*}

Now suppose that in Theorem \ref{thm:HK:def:HK:cones} the real section $s$ of $T_X(2)$ is of the form $s=q(\xi)$ for a section $\xi$ of $\Lie{h}^\C\otimes\mathcal{O}(2)$ such that there is no $\alpha\in R$ with $\alpha (\xi)=0\in H^0(\PP^1;\mathcal{O}(2))$. Using the graded Poisson deformation $\pi'\co \Lie{X}'\ra \Lie{h}^\C$ we then produce a twistor space $Z'_\xi$ as in Proposition \ref{prop:Twistor:space}.

\begin{conjecture}
	$Z'_\xi$ a singular model for the twistor space of a complete hyperk\"ahler metric on $\widetilde{X}$ given by the twistor construction of Theorem \ref{thm:HK:def:HK:cones} over a dense open set. 
\end{conjecture}
For all other choices of $s$ we expect that the hyperk\"ahler metric $g_s$ produced by Theorem \ref{thm:HK:def:HK:cones} cannot be completed to a hyperk\"ahler metric on a smooth manifold.

\begin{remark*}
It is expected that in (real) dimension $4n>4$ the only isolated symplectic singularity admitting a symplectic resolution is the minimal nilpotent orbit of $\Lie{sl}(n+1,\C)$, \cf \cite[Section 2.3]{Fu}, with Calabi's hyperk\"ahler metrics on $T^\ast\PP^{n}$ \cite{Calabi} as asymptotically conical hyperk\"ahler desingularisations. Thus complete maximal volume growth hyperk\"ahler metrics arising from Theorem \ref{thm:HK:def:HK:cones} will generally have singular tangent cone at infinity.  
\end{remark*}

\section{Examples and applications}

The construction of the previous sections applies uniformly to any hyperk\"ahler cone, including quiver and hypertoric varieties, nilpotent orbits and quotients $\HH^n/\Gamma$ for finite subgroups $\Gamma$ of $\tu{Sp}(n)$. The drawback of our construction (and of twistor methods more generally) is that we can detect which metrics produced by Theorem \ref{thm:HK:def:HK:cones} are complete only in cases where we know them in advance! In this section we discuss two  classes of examples where also the new incomplete metrics produced by Theorem \ref{thm:HK:def:HK:cones} are of interest. We study deformations of regular nilpotent orbits, which have applications to hyperk\"ahler manifolds with symmetries and hyperk\"ahler quotients, and deformations of Kleinian singularities $\C^2/\Gamma$. In the latter case, we explain how the (incomplete) metrics produced by Theorem \ref{thm:HK:def:HK:cones} in addition to Kronheimer's ALE metrics on the minimal resolution of $\C^2/\Gamma$ can be used to 
 interpret a duality in physics between codimension-$4$ singularities of $\gtwo$--holonomy metrics and certain gauge theoretic equations on $3$-manifolds.

\subsection{The nilpotent cone of a simple Lie algebra}

Let $\Lie{g}$ be a simple Lie algebra of rank $r$ and denote by $G$ the corresponding simply connected compact Lie group. We use the Killing form $\langle\,\cdot\, , \,\cdot\,\rangle_{\Lie{g}}$ of $\Lie{g}$ to identify $\Lie{g}$ and its dual. Let $\Lie{h}$ be a Cartan subalgebra of $\Lie{g}$ and let $W$ be the Weyl group acting on $\Lie{h}$. Denote by $\Lie{g}^\C$, $G^\C$ and $\Lie{h}^\C$ the complexifications of $\Lie{g}$, $G$ and $\Lie{h}$, respectively.

Let $X=\mathcal{N}$ be the nilpotent cone of $\Lie{g}^\C$, \ie the closure of the adjoint orbit of a regular nilpotent element. Then $X$ is a conical hyperk\"ahler variety in the sense of Definition \ref{def:Singular:HK:Cones}: the $\C^\ast$--action is induced by standard scaling on the vector space $\Lie{g}^\C$, the real structure is the obvious one induced from the real form $\Lie{g}$ of $\Lie{g}^\C$, the holomorphic symplectic form is the Kirillov--Konstant--Souriau symplectic form on (co)adjoint orbits and the hyperk\"ahler cone metric was constructed by Kronheimer \cite{Kronheimer:Nilpotent} using moduli spaces of Nahm's equations on a semi-infinite interval.

Namikawa's universal Poisson deformation of $X$ is the adjoint quotient $\pi\co \Lie{X}=\Lie{g}^\C \ra T_X=\Lie{h}^\C/W\simeq \C^r$ \cite{Lehn:Namikawa:Sorger}, where the Poisson structure on $\Lie{g}^\C$ is induced by the Lie bracket and the Killing form. Note that the weights $d_1,\dots, d_r$ of the $\C^\ast$--action on $T_X$ are the degrees of a basis of Ad-invariant polynomials. It is also known that $X$ admits a unique symplectic resolution, the Springer resolution $T^\ast (G^\C/B)$, where $B$ is a Borel subgroup of $G^\C$. Then diagram \eqref{eq:Symplectic:Resolution} is Grothendieck's simultaneous resolution
 \[
\begin{CD}
\widetilde{\Lie{g}}^\C @>>> \Lie{g}^\C\\
@V\widetilde{\pi}VV  @VV\pi V\\
\Lie{h}^\C @>q>> \Lie{h}^\C/W
\end{CD}
\] 
where $\widetilde{\Lie{g}}^\C=G^\C\times_B \Lie{b}$ with $\Lie{b}$ the Lie algebra of $B$.

The fibre bundle $\Lie{X}(2)\ra\PP^1$ is simply the vector bundle $\Lie{g}^\C\otimes\mathcal{O}(2)$, so real sections $m$ of $\Lie{X}(2)$ are identified with triples $(T_1,T_2,T_3)\in \Lie{g}\otimes\R^3$. Similarly, real sections of $\Lie{h}^\C\otimes\mathcal{O}(2)$ are triples $\xi=(\xi_1,\xi_2,\xi_3)\in\Lie{h}\otimes\R^3$. By a result of Konstant \cite[Theorem 0.1]{Kostant:Polynomial:Rings}, an element of $\Lie{X}=\Lie{g}^\C$ is regular in the sense of Definition \ref{def:Regular:Poisson:def} when it is regular in the usual sense that its Lie algebra centraliser has minimal dimension $r$. Theorem \ref{thm:HK:def:HK:cones} yields hyperk\"ahler structures parametrised by real sections $s$ of $\bigoplus_{i=1}^r\mathcal{O}(2d_i)$. The metrics corresponding to $s=q(\xi)$ for a triple $\xi\in \Lie{h}\otimes\R^3$ are the Kronheimer--Biquard--Kovalev metrics on regular complex adjoint orbits \cite{Biquard:Nahm,Kovalev:Nahm,Kronheimer:Coadjoint}, also constructed using Nahm's equations. They are complete only when $\xi$ is regular \cite[Proposition 2]{Bielawski:orbits}, meaning that its Lie algebra centraliser reduces to $\Lie{h}$. The other metrics produced by Theorem \ref{thm:HK:def:HK:cones} are due to D'Amorim Santa-Cruz \cite{D'Amorim:SantaCruz} (and were conjectured by Joyce \cite[\S 12]{Joyce:Hypercomplex}), but their signature was not known in (real) dimension larger than $4$.    

These metrics are interesting for a number of reasons. First of all, all these metrics admit a tri-Hamiltonian action of the compact Lie group $G$. The hyperk\"ahler moment map is simply the identification of elements of $M_X(s)$ with triples $(T_1,T_2,T_3)\in \Lie{g}\otimes\R^3$. Furthermore, in any complex structure the space $M_X(s)^+$ is identified with an open set in a regular complex adjoint orbit. In \cite{Bielawski:orbits} the first author of this paper showed that every complete $G$--invariant hyperk\"ahler metric with a homogeneous complex structure must be one of the complete Kronheimer--Biquard--Kovalev metrics. The homogeneity of the complex structure means that every point of a dense open set has a neighbourhood biholomorphic to a regular adjoint $G^\C$--orbit. In particular, all other metrics produced by Theorem \ref{thm:HK:def:HK:cones} are incomplete. Note that any $G$--invariant hyperk\"ahler metric (complete or not) with the additional property above must arise from our twistor construction since the fibres of its twistor space must be adjoint orbits exactly as the one produced by Proposition \ref{prop:Twistor:space} in this setting. A classification of all $G$--invariant metrics with homogeneous complex structure therefore requires an understanding of all the connected components of the spaces $M_X(s)$ of regular twistor lines on which $g_s$ restricts to a positive definite hyperk\"ahler metric.
\par
It turns out that the structure of the connected components of $M_X(s)$ is more complicated than expected:
\begin{theorem}\label{thm:indefinite:components}
Let $X$ be the nilpotent cone of a simple complex Lie algebra $\Lie{g}^\C$.
\begin{enumerate} 
\item[(i)] $M_X(s)$ has a component on which $g_s$ is positive-definite, and a component on which $g_s$ is negative-definite.
\item[(ii)] For $\Lie{g}=\Lie{su}(3)$ there exists an open dense subset of real sections $s$ of $\mathcal{O}(4)\oplus\mathcal{O}(6)$ such that $M_X(s)$ has a component on which $g_s$ is indefinite.
\end{enumerate}
\end{theorem}
\proof
Part (i) follows immediately from Theorem \ref{thm:HK:def:HK:cones} and the observation that $(T_1,T_2,T_3)\mapsto -(T_1,T_2,T_3)$ sends the positive definite component $M_X(s)^+$ to a component $M_X(s)^-$ on which $g_s$ is negative definite. We shall prove part (ii) in a series of lemmata, only the last one of which is special to the case $\Lie{g}=\Lie{su}(3)$.

First of all, it will be convenient to pass to the complexification $\Lie{g}^\C\otimes\C^3$ of $\Lie{g}\otimes\R^3$. We identify $\Lie{g}^\C\otimes\C^3$ with $H^0(\PP^1;\Lie{g}^\C\otimes\mathcal{O}(2))$ via $(A_0,A_1,A_2)\mapsto A(\zeta)=A_0 + A_1\zeta+A_2\zeta^2$. The triple $(A_0,A_1,A_2)$ is real if $A_0=T_2+iT_3$, $A_1 = 2iT_1$ and $A_2=T_2-iT_3$ for a triple $(T_1,T_2,T_3)\in\Lie{g}\otimes\R^3$.

Now, in the current setting the map in \eqref{eq:Strongly:regular:twistor:line} becomes $d_A\pi\co \Lie{g}^\C\ra \bigoplus_{i=1}^r{H^0(\PP^1;\mathcal{O}(2d_i-2))}$, where $\pi$ is given by evaluation of a basis of Ad-invariant polynomial at $A(\zeta)$. According to Lemma \ref{lem:Regular:twistor:line} we are interested in triples $(A_0,A_1,A_2)$ such that $\det{d_A\pi}\neq 0$. Whenever this is the case, the Lie algebra element $A(\zeta)$ is automatically regular, \ie $d_{A(\zeta)}\pi\co \Lie{g}^\C\ra \C^r$ is surjective, for all $\zeta\in \PP^1$: indeed, given $\zeta$ and $u\in\C^r$, simply choose any section $\phi$ of $\bigoplus_{i=1}^r\mathcal{O}(2d_i-2)$ with $\phi(\zeta)=u$. We conclude that regular twistor lines correspond to triples $(T_1,T_2,T_3)\in \Lie{g}\otimes\R^3\setminus D_1^\R$, where
 \begin{subequations}
 \begin{equation}\label{D2}
 D_1 = \{ (A_0,A_1,A_2)\in \Lie{g}^\C\otimes\C^3\, |\, \det{d_A\pi}=0\}.
 \end{equation}	
 and $D_1^\R$ is the intersection of this hypersurface with $\Lie{g}\otimes\R^3$.
 
 We now introduce a second hypersurface
 \begin{equation}\label{D_1}
 D_2=\{(A_0,A_1,A_2)\in\Lie{g}^\C\otimes\C^3\,|\, \langle A_1,[A_0,A_2]\rangle=0\}.
 \end{equation}
 \end{subequations}
The significance of $D_2$ is explained by the following lemma.  

\begin{lemma} If $(T_1,T_2,T_3)\in M_X(s)$ satisfies  $\langle T_1,[T_2,T_3]\rangle_{\Lie{g}}=0$, then the pseudo-hyperk\"ahler metric $g_s$ is indefinite at $(T_1,T_2,T_3)$.
\end{lemma}
\begin{proof} Let $(\mu_1,\mu_2,\mu_3)$ be the pseudo-hyperk\"ahler moment map for the $G$--action on $M_X(s)$. The proof of \cite[Lemma 2.1]{Bielawski:orbits} shows that
\[
J_1Y_{\mu_1(m)}=J_2Y_{\mu_2(m)}=J_3Y_{\mu_3(m)},
\]
where $Y_{\rho}$ denotes the fundamental vector field on $M_X(s)$ corresponding to $\rho\in \Lie{g}$. If $m=(T_1,T_2,T_3)$ satisfies the assumption of the Lemma then the vector $Y=J_1Y_{\mu_1(m)}$ is null at $m$ since
\[
g_s(Y,Y)|_m=g_s(J_1Y_{\mu_1(m)},J_2Y_{\mu_2(m)})|_m=\big\langle d_m\mu_3(Y_{\mu_1(m)}),\mu_2(m)\big\rangle_{\Lie{g}}=\big\langle[\mu_1(m),\mu_3(m)],\mu_2(m)\big\rangle_{\Lie{g}}=0,
\]
where we used the relation between $g_s$, $J_3$ and $\omega_3$, the definition of moment map and \cite[(2.4)]{Bielawski:orbits}. 

On the other hand $Y$ does not vanish at $m$ since otherwise $Y_{\mu_1(m)}=Y_{\mu_2(m)}=Y_{\mu_3(m)}=0$, which, using \cite[(2.3)]{Bielawski:orbits},  means that the triple $(T_1,T_2,T_3)$ is commuting, contradicting its regularity. 
\end{proof}

Owing to the above lemma, part (ii) of Theorem \ref{thm:indefinite:components} is equivalent to $D_2^\R\backslash D_1^\R$ being nonempty.

\begin{lemma} The variety $D_2$ is irreducible.\end{lemma}
\begin{proof} $D_2$ is defined by a homogeneous polynomial and hence $D_2\backslash\{0\}$ is a $\C^\ast$--bundle over a hypersurface in a projective space. Therefore $D_2$ is connected. Singular points of $D_2$ correspond to commuting triples $(A_0,A_1,A_2)$, \ie triples in the kernel of the map
\begin{equation}\label{eq:Commuting:Triples}
(A_0,A_1,A_2)\longmapsto \bigl([A_1,A_2],[A_2,A_0],[A_0,A_1]\bigr).
\end{equation}
Let $\tu{Jac}\co \Lie{g}^\C\otimes\C^3\ra \Lie{g}^\C\otimes\C^3$ denote the Jacobian of this map at a nonzero commuting triple $(A_0,A_1,A_2)$. It suffices to show that the rank of $\tu{Jac}$ is at least $4$, for then the set of singular points of $D_2$ has codimension at least $3$ and $D_2$ must be irreducible. 

In order to study the kernel of $\tu{Jac}$, denote by $\Lie{z}$ the common centraliser of $(A_0,A_1,A_2)$ and observe that $\Lie{z}$ is the kernel of $\xi\mapsto \sum_{i=0}^2{[A_i,[A_i,\xi]]}$. Since \eqref{eq:Commuting:Triples} is $G^\C$--invariant, $([A_0,\xi],[A_1,\xi],[A_2,\xi])\in\ker\tu{Jac}$ for any $\xi\in\Lie{g}^\C$. It is then natural to introduce the gauge fixing condition $\sum_{i=0}^2{[A_i,a_i]}=0$. Noting that $\tu{Jac}^2 = -\sum_{i=0}^2{[A_i,[A_i,\,\cdot\,]]}\otimes\tu{id}_{\C^3}$ on gauged fixed triples, we conclude that $\dim\ker\tu{Jac} = 3\dim\Lie{z} + (\dim\Lie{g}^\C-\dim\Lie{z})$ and therefore $\tu{rk}\,\tu{Jac} = 2(\dim\Lie{g}^\C-\dim\Lie{z})\geq 4$ as claimed since $\dim\Lie{g}^\C-\dim\Lie{z}$ is at least the dimension of the minimal non-zero adjoint orbit, which, adjoint orbits being symplectic, is at least $2$. 
\end{proof}

\begin{lemma} If $D_2^\R\subseteq D_1^\R$, then $D_2\subseteq D_1$.\end{lemma}
\begin{proof} 
Let $p_1,p_2$ denote the defining polynomials of $D_1$ and $D_2$, and $p_1^\R,p_2^\R$ their restriction to $\Lie{g}\otimes\R^3$. Since $p_2$ is irreducible,  so is $p_2^\R$. Moreover, since the restriction of $p_2^\R$ to any line in $\Lie{g}\otimes\R^3$ is a polynomial of degree $3$, $p_2^\R$ changes sign on $\Lie{g}\otimes\R^3$.  This implies that  the ideal $I$ of polynomials vanishing on $D_2^\R$ is equal to  $\langle p_2^\R\rangle$ (cf.\ \cite[Theorem 4.5.1]{Bochnak}). Owing to the assumption, $p_1^\R\in I$, hence $p_1^\R|p_2^\R$ and $p_1|p_2$.
\end{proof}

\begin{lemma} For $\Lie{g}=\Lie{su}(3)$, $D_2\not\subset D_1$.\end{lemma}
\begin{proof} Let $\epsilon$ be a $3$rd root of unity  and decompose $\Lie{sl}(3,\C)$ into eigenspaces $V_1,V_\epsilon, V_{\epsilon^2}$ of the operator $\Ad \diag(1,\epsilon,\epsilon^2)$. Note that $V_1$ is the maximal torus of diagonal matrices.

 A triple $(A_0,A_1,A_2)$ with $A_1\in V_1$, $A_0,A_2\in V_1\oplus V_\epsilon$ automatically belongs to $D_2$. Noting that
 \[
 d_{A(\zeta)}\pi = \left( \langle A(\zeta),\, \cdot\, \rangle_{\Lie{sl}(3)}, \langle A(\zeta)^2,\, \cdot\, \rangle_{\Lie{sl}(3)}\right),
 \]
we need to find such a triple with $A_0,\dots,A_7$ linearly independent, where
\[
A_3=A_0A_1+A_1A_0,\qquad A_4= A_2A_1+A_1A_2,\qquad A_5=A_0^2,\qquad A_6=A_2^2,\qquad A_7=A_1^2+A_0A_2+A_2A_0,
\]
are the coefficients of $A(\zeta)^2$.

Suppose that $\sum_{i=0}^7\lambda_i A_i=0$ is a relation. Owing to the assumption on $A_0,A_1,A_2$, only $A_5,A_6,A_7$ have entries in $V_{\epsilon^2}$ and these entries are quadratic monomials in the entries of $A_0$ and $A_2$ lying in $V_\epsilon$. Generically, therefore, the  $V_{\epsilon^2}$--parts of $A_5,A_6,A_7$ are linearly independent, and hence $\lambda_5=\lambda_6=\lambda_7=0$. The equation on the $V_\epsilon$--part is then
\[
\lambda_0A_0^\epsilon+\lambda_2A_2^\epsilon+\lambda_3A_3^\epsilon+\lambda_4A_4^\epsilon=0,
\]
where $A_i^\epsilon$ denotes the $V_\epsilon$--part of $A_i$. Thus generically $(\lambda_0,\lambda_2,\lambda_3,\lambda_4)$ belong to a $1$-dimensional subspace. Consequently, denoting similarly by $A_i^1$ the $V_1$-part of each $A_i$,
\[
\lambda_0A_0^1+\lambda_2A_2^1+\lambda_3A_3^1+\lambda_4A_4^1
\]
spans a $1$-dimensional subspace of $V_1$ and, if we choose $A_1^1=A_1$ not in this subspace, then $\lambda_i=0$ for each $i=0,\dots,7$.
\end{proof}

Given the genericity of triples in $D_2\setminus D_1$, part (ii) of Theorem \ref{thm:indefinite:components} now follows from the above four lemmata.
\endproof

\subsection{Hyperk\"ahler quotients at arbitrary level\label{arbitary;level}} If $M$ is a symplectic manifold with a Hamiltonian action of a Lie group $G$ and moment map $\mu$, then one can form a symplectic quotient $\mu^{-1}(l)/\tu{Stab}_G(l)$ at any level $l\in \Lie{g}^\ast$ \cite{Guillemin:Sternberg}. This symplectic quotient is isomorphic to the symplectic quotient of the product $M \times \left( G\cdot(-l)\right)$ at level $0$, where $G\cdot(-l)$ is the coadjoint orbit of $-l$ with its homogeneous Kirillov--Kostant--Souriau symplectic form. In other words $\mu^{-1}(l)/\tu{Stab}_G(l)$ can be identified with $\mu^{-1}(G\cdot l)/\tu{Stab}_G(l)$. 
\par
Theorem \ref{thm:indefinite:components} allows us to define hyperk\"ahler analogues of symplectic quotients at arbitrary levels. Let $M$ be a hyperk\"ahler manifold with a tri-Hamiltonian action of a compact Lie group $G$ and let $(l_1,l_2,l_3)\in \Lie{g}^\ast\otimes\R^3$. Since the hyperk\"ahler quotient can be done in stages for the centre of $G$ and for its semisimple part, we can assume without loss of generality that $G$ is semisimple. We can then identify $\Lie{g}^\ast$ with $\Lie{g}$ using the Killing form. The triple $(l_1,l_2,l_3)$ determines a real section $l(\zeta)=(l_2+il_3) + 2il_1 \zeta +(l_2-il_3)\zeta^2$ of $\Lie{g}^\C\otimes\mathcal{O}(2)$ and we assume that $l(\zeta)\in\Lie{g}^\C$ is regular for every $\zeta\in\PP^1$. Consider the real section $s$ of  $\Lie{g}^\C\otimes\mathcal{O}(2)/G^\C\simeq \bigoplus_{i=1}^r \mathcal{O}(2d_i) $ determined by $-l(\zeta)$. We now define the hyperk\"ahler quotient of $M$ at level $(l_1,l_2,l_3)$ as the usual hyperk\"ahler quotient of $M\times M_X(s)^+$ at level $0$, where $X$ is the nilpotent cone of $\Lie{g}^\C$.
The definition is justified by the fact that the twistor space of the hyperk\"ahler quotient of  $M\times M_X(s)^+$ is the fibrewise complex symplectic quotient of the twistor space of $M$ at level $l(\zeta)$, $\zeta\in\PP^1$. We emphasize that we need a positive-definite component of $M_X(s)$ for this construction to work: the pseudo-hyperk\"ahler quotient of $M\times M_X(s)$ will usually exhibit various pathologies.

\begin{remark*}
When $s=q(-\xi)$ for a regular commuting triple $\xi$, the hyperk\"ahler quotient of $M\times M_X(s)^+$ coincides with Kronheimer's definition of the hyperk\"ahler quotient of $M$ at level $\xi$ \cite[\S 4.(b)]{Kronheimer:Coadjoint}. While Kronheimer's definition appears to depend on the triple $\xi$ itself and not only on its $W$--orbit, \ie the associated section $s$, note in fact that Kronheimer's metrics on regular complex adjoint orbits parametrised by regular commuting triples in the same $W$--orbit are isomorphic as hyperk\"ahler manifolds: in terms of Nahm's equations these isomorphisms are realised by the action of gauge transformations asymptotic at infinity to arbitrary elements in the normaliser of the maximal torus, instead of the torus itself.
\end{remark*}

The application of Theorem \ref{thm:HK:def:HK:cones} to deformations of (the normalisation of) a more general nilpotent orbit closure $X$ instead of the nilpotent cone of $\Lie{g}^\C$ could be used to define hyperk\"ahler quotients at a more general (non-regular) level $(l_1,l_2,l_3)$ in a similar way. The standard hyperk\"ahler quotient of $M$ at level zero corresponds to the choice $X=\{ 0\}$.

\subsection{Kleinian singularities}
We now consider a general $4$-dimensional hyperk\"ahler cone, which must be of the form $X=\C^2/\Gamma$ for a finite subgroup $\Gamma$ of $\sunitary{2}$. Via the McKay correspondence each of these groups is associated to a simply laced Dynkin diagram and therefore a simple Lie algebra of type A, D or E. By work of Grothendieck, Brieskorn and Slodowy the McKay correspondence can also be used to describe the deformations of $X$, cf.\ \cite{Slodowy}.

Let $\Lie{g}$ a simple Lie algebra of type A, D or E. The nilpotent cone $\mathcal{N}$ in $\Lie{g}^\C$ does not have an isolated singularity. In particular, $\mathcal{N}$ has (real) codimension-$4$ singularities along the subregular nilpotent orbit: the transverse singularity type is precisely $X=\C^2/\Gamma$ for the finite subgroup $\Gamma$ corresponding to $\Lie{g}$ via the McKay correspondence. More precisely, fix an element $e$ in the subregular nilpotent orbit and denote by $(h,e,f)$ the associated $\Lie{sl}(2)$--triple in $\Lie{g}^\C$. Let $S_e$ be the Slodowy slice to the subregular nilpotent orbit through $e$, \ie $S_e = e + \ker{[f,\,\cdot\,]}$. Then $S_e\cap\mathcal{N}=X$ and $\pi\co \Lie{X}=S_e \hookrightarrow \Lie{g}^\C\ra \Lie{h}^\C/W=T_X$ is the universal graded Poisson deformation of $X$ \cite{Lehn:Namikawa:Sorger}.   

Theorem \ref{thm:HK:def:HK:cones} yields hyperk\"ahler structures parametrised by real sections $s$ of $\bigoplus_{i=1}^r\mathcal{O}(2d_i)$, where $d_1,\dots, d_r$ are the degrees of a basis of $\tu{Ad}$-invariant polynomials on $\Lie{g}$. Since in this case $X$ is smooth outside the vertex $o$, the discussion after Theorem \ref{thm:HK:def:HK:cones} implies that all these metrics are ALE asymptotic to $\C^2/\Gamma$. 

\begin{theorem}\label{thm:ALE}
Let $\Gamma$ be a finite subgroup of $\sunitary{2}$ and denote by $d_1,\dots, d_r$ the degrees of a basis of Ad-invariant polynomials on the simple Lie algebra $\Lie{g}$ corresponding to $\Gamma$ via the McKay correspondence. Then for every real section $s$ of $\bigoplus_{i=1}^r{H^0(\PP^1;\mathcal{O}(2d_i))}$ there exists an ALE hyperk\"ahler metric $g_s$ defined on a large enough exterior domain in $\C^2/\Gamma$.	
\end{theorem}

The minimal resolution $\widetilde{X}$ of $X$ is $\widetilde{S}_e\cap\widetilde{\Lie{g}}^\C$, where $\widetilde{S}_e$ is the preimage of $S_e$ in the Grothendieck simultaneous resolution $\widetilde{\Lie{g}}^\C\ra\Lie{g}^\C$ of the adjoint quotient. The universal graded Poisson deformation of $\widetilde{X}$ is $\widetilde{S}_e\ra\Lie{h}^\C$ and Kronheimer's ALE metrics \cite{Kronheimer:ALE:Construction} on $\widetilde{X}$ coincide with the hyperk\"ahler metrics of Theorem \ref{thm:ALE} when $s=q(\xi)$ for a generic triple $\xi\in\Lie{h}\otimes\R^3$, interpreted as a real section of $\Lie{h}^\C\otimes\mathcal{O}(2)$. In particular, Theorem \ref{thm:ALE} yields a purely twistorial construction of Kronheimer's metrics for all $\Gamma$. Since Kronheimer's metrics are the unique complete ALE metrics \cite{Kronheimer:ALE:Classification}, all other choices of $s$ in Theorem \ref{thm:ALE} yield incomplete ALE ends.

\begin{remark*}
The hyperk\"ahler manifolds of Theorem \ref{thm:ALE} are hyperk\"ahler Slodowy slices \cite{Bielawski:Slodowy:Slices,Bielawski:Slices:Sums} to the D'Amorim Santa-Cruz's manifolds of Theorem \ref{thm:indefinite:components}(i), \ie with respect to any complex structure compatible with the metric $M_{\C^2/\Gamma}(s)^+=\mu^{-1}(S_e)\subset M_\mathcal{N}(s)^+$ for the holomorphic symplectic moment map $\mu\co M_\mathcal{N}(s)^+\ra\Lie{g}^\C$. The fact that Kronheimer's ALE metrics are Slodowy slices to regular adjoint orbits is \cite[Theorem 2]{Bielawski:ALE}. 	
\end{remark*}

\subsubsection{An explicit example}

In the simplest case $X=\C^2/\Z_2$ we can be extremely explicit. In this case the cone metric admits a triholomorphic $\sunitary{2}$--action inherited by all metrics produced by Theorem \ref{thm:ALE} (in fact $X$ is the nilpotent cone of $\Lie{sl}(2,\C)$). The large symmetry group reduces the problem to the study of an ODE system, the solutions of which are explicitly known \cite{Belinskii}.

Theorem \ref{thm:ALE} yields hyperk\"ahler structures parametrised by a real section of $\mathcal{O}(4)\ra\PP^1$. Under the natural action of $\sorth{3}=\tu{P}\unitary{2}$ on $\PP^1$, the parameter space $H^0(\PP^1;\mathcal{O}(4))^\R$ is identified with the irreducible representation $\tu{Sym}^2_0(\R^3)$ of $\sorth{3}$. This action of $\sorth{3}$ corresponds to a hyperk\"ahler rotation of the hyperk\"ahler structure, so the metrics produced by Theorem \ref{thm:ALE} are in fact parametrised by $\tu{Sym}^2_0(\R^3)/\sorth{3} \simeq \{ (\lambda_1,\lambda_2,\lambda_3)\in \R^3\, |\, \lambda_1+\lambda_2+\lambda_3=0, \lambda_1\leq \lambda_2\leq \lambda_3\}$. The metrics are explicit \cite{Belinskii}: up to a reparametrisation $(\lambda_1,\lambda_2,\lambda_3)\mapsto (a_1,a_2,a_3)$ with $0=a_1\leq a_2\leq a_3$,
\[
g_{\lambda_1,\lambda_2,\lambda_3} = \frac{1}{\sqrt{\Lambda_1(r)\Lambda_2(r)\Lambda_3(r)}} \left( dr^2 + \tfrac{1}{4}r^2\sum_{i=1}^3{\Lambda_j(r) \Lambda_k(r)\, \sigma_i^2}\right) 
\]  
where $\Lambda_i(r) = 1-\frac{16a_i}{r^4}$ and $(\sigma_1,\sigma_2,\sigma_3)$ is the standard basis of left-invariant $1$-forms on $\sorth{3}$. The metric is a priori only defined for $r^4 > 16 a_3$ and is asymptotic to the flat orbifold metric $g_{0,0,0}$ on $\C^2/\Z_2$ as $r\ra\infty$. If $0<a_2=a_3$ the metric $g_{\lambda_1,\lambda_2,\lambda_3}$ extends smoothly at $r^4 = 16 a_3$ over a lower dimensional $\sorth{3}$--orbit $S^2$ and coincides with the Eguchi--Hanson metric on $T^\ast S^2$. All other metrics are incomplete: for example, if $0=a_2<a_3$ then as $r^4 \ra 16 a_3$
\[
\sqrt{\tfrac{3}{2}a_3}\,  g_{\lambda_1,\lambda_1,\lambda_2} \approx  d\rho^2 + \tfrac{2}{3}\rho^{\frac{3}{2}}(\sigma_1^2+\sigma_2^2) + \rho^{-\frac{3}{2}}\sigma_3^2, 
\]   
where we define $\rho$ by $r^4 = 16 a_3\left(1+\tfrac{2}{3}\rho^3\right)$.

\begin{remark*}
In the context of cohomogeneity one manifolds, construction of families of asymptotically conical ends which are backward complete only for special choices of parameters is not unique to this $4$-dimensional hyperk\"ahler case: see for example \cite{Dancer:Wang} in the Einstein case and \cite{FHN} in the $\gtwo$--holonomy setting.  	
\end{remark*}

Finally, we give another interpretation of the parameter space $\tu{Sym}^2_0(\R^3)$ of the hyperk\"ahler metrics produced by Theorem \ref{thm:ALE} in terms of infinitesimal deformations of the conical hyperk\"ahler structure $\triple{\omega}_{\tu{C}}=(\omega_1,\omega_2,\omega_3)$ on $\C^2/\Z_2$. Note that for any harmonic function $f$ on $\C^2/\Z_2$, the $2$-forms $dd^\ast (f\omega_i)$, $i=1,2,3$, are anti-self-dual and that in dimension $4$ triples of anti-self-dual $2$-forms are infinitesimal hyperk\"ahler deformations. The only $\sunitary{2}$--invariant decaying harmonic function on $\C^2/\Z_2$ is the Green's function and therefore we obtain a $9$-parameter family of infinitesimal hyperk\"ahler deformations. The parameter space is naturally identified with $\tu{Mat}_{3,3}(\R)$ and the isometric action of $\sorth{3}$ that rotates the complex structures acts on this space by conjugation. A $4$-dimensional subfamily of these infinitesimal deformations actually coincides with Lie derivatives (note that $dd^\ast (f\omega_i)=\mathcal{L}_{\nabla f}\omega_i$), thus leaving a $5$-parameter subfamily, naturally identified with $\tu{Sym}^2_0(\R^3)$, of genuine infinitesimal hyperk\"ahler deformations: Theorem \ref{thm:HK:def:HK:cones} (and the explicit description in \cite{Belinskii}) shows that these infinitesimal deformations can all be integrated to hyperk\"ahler structures on large enough exterior domains in $\C^2/\Z_2$.

\subsubsection{Application to codimension-$4$ singularities of $\tu{G}_2$--holonomy metrics}\label{sec:3dHitchin}

We conclude with a brief discussion of a possible application of the incomplete ALE ends of Theorem \ref{thm:ALE} to the study of local deformations of special holonomy metrics with codimension-$4$ singularities and the expectation in theoretical physics that these deformations have a dual gauge theory description.

The model for this expectation is a result of Diaconescu--Donagi--Pantev \cite{Diaconescu:Donagi:Pantev} which relates the moduli space of Higgs bundles on a Riemann surface to certain non-compact complex $3$-folds with trivial canonical bundle. More precisely, Diaconescu--Donagi--Pantev consider a holomorphic fibration of $\C^2/\Gamma$--singularities along a Riemann surface $C$ so that the total space $\mathcal{Y}_0$ has trivial canonical bundle. Using the versal deformation of the singularity $\C^2/\Gamma$ they then construct a family of $3$-folds $\mathcal{Y}_q$ with trivial canonical bundle parametrised by holomorphic sections $q$ of $\bigoplus_{i=1}^r{K_C^{d_i}}$, where $K_C$ is the canonical bundle of $C$ and $d_1,\dots, d_r$ are as in Theorem \ref{thm:ALE}. Interpreting $\bigoplus_{i=1}^r{H^0(C;K_C^{d_i})}$ as the base of the Hitchin fibration on the moduli space of $G^\C$--Higgs bundles on $C$, Diaconescu--Donagi--Pantev further identify Hitchin's integrable system on the moduli space of Higgs bundles and the integrable system of intermediate Jacobians of the family $\mathcal{Y}$ away from their discriminant.

Pantev--Wijnholt \cite{Pantev:Wijnholt} introduced a similar conjectural picture in the context of codimension-$4$ singularities of $7$-dimensional metrics with holonomy $\tu{G}_2$.

On the geometry side, we consider a $\tu{G}_2$--holonomy metric with singularities along a smooth (associative) submanifold $Q^3$ with transverse singularity modelled on $\C^2/\Gamma$. More precisely, we only retain the information of the leading order behaviour of the $\tu{G}_2$--holonomy metric near the singular locus $Q$, which we assume is modelled as follows. On the $3$-manifold $Q$ we are given a Riemannian metric $g_Q$ and a spin structure $\mathcal{S}$ endowed with the connection induced by the Levi-Civita connection of $g_Q$. We then consider the fibration $\mathcal{W}_0=\mathcal{S}\times_{\sunitary{2}}\C^2/\Gamma$, where $\sunitary{2}$ acts on $\C^2/\Gamma$ via the isometric action that rotates the complex structures (if $\Gamma$ is a cyclic group, then we can exploit the triholomorphic $\unitary{1}$--action on $\C^2/\Gamma$ to consider more generally a fibration $\widetilde{\mathcal{S}}\times_{\unitary{2}}\C^2/\Gamma$, where $\widetilde{\mathcal{S}}$ is a $\tu{Spin}^\C$--s
 tructure on $Q$ with any connection inducing the Levi-Civita connection of $g_Q$ on the frame bundle $\widetilde{\mathcal{S}}\times_{\unitary{2}}\tu{PU}(2)$). Then on $\mathcal{W}_0$ we consider the Riemannian metric $g_{\mathcal{W}_0}=g_Q + g_{\C^2/\Gamma}$, where we implicitly use the connection on $\mathcal{S}$. (Note that this metric does not have holonomy $\tu{G}_2$ unless $Q$ is flat, but it models the leading-order behaviour of singular $\tu{G}_2$--holonomy metrics on $\mathcal{W}_0$ as we approach the singular set $Q$.)

On the gauge theory side, Pantev--Wijnholt consider the system describing stable flat $G^\C$--connections on $Q$:
\begin{equation}\label{eq:PW:system}
	F_A-[\phi,\phi]=0, \qquad d_A\phi=0=d_A^\ast\phi.
\end{equation}
 Here $G$ is the simple Lie group associated with $\Gamma$ via the McKay correspondence, $A$ is a connection on a principal $G$--bundle $P\ra Q$ and $\phi\in\Omega^1(Q;\ad P)$.
 
 The expectation in \cite{Pantev:Wijnholt} is that the moduli space of solutions to \eqref{eq:PW:system} encodes the data of local $\gtwo$--holonomy deformations of $g_{\mathcal{W}_0}$ (plus additional information encoded in the moduli of a $C$--field which complexifies the $\gtwo$ moduli space). When $[\phi,\phi]=0$ we can make sense of this expectation, at least at a formal level, using Donaldson's description of adiabatic limits of coassociative fibrations \cite{Donaldson:Adiabatic:Coassociatives}. Indeed, in this case locally in $Q$ we can write $\phi=\sum_{i=1}^3{\phi_i\, dx_i}$ for a map $\triple{\phi}=(\phi_1,\phi_2,\phi_3)\co Q\ra \Lie{h}\otimes\R^3$. Interpreting the target as the parameter space of Kronheimer's ALE metrics on the minimal resolution of $\C^2/\Gamma$, we can now desingularise $g_{\mathcal{W}_0}$ by replacing the $\C^2/\Gamma$--fibres with the hyperk\"ahler metrics parametrised by $\triple{\phi}$. The original equations satisfied by $\phi$ imply that the map $
 \triple{\phi}$ satisfies Donaldson's adiabatic equations. For example, in the simplest case $\C^2/\Z_2$, $\triple{\phi}$ is just a (nowhere vanishing) harmonic $1$-form and Joyce--Karigiannis \cite{Joyce:Karigiannis} use the resulting approximate $\gtwo$--metric as a building block in a desingularisation procedure of $\gtwo$--orbifolds $M/\Z_2$. For a discussion in the case of general $\Gamma$ see also \cite{Barbosa}.
 
The new hyperk\"ahler metrics of Theorem \ref{thm:ALE} allow us to extend this formal picture also to the case where $[\phi,\phi]\neq 0$. As a first step, we extend the definition of Hitchin's map on the moduli space of Higgs bundles on a Riemann surface to the $3$-dimensional system \eqref{eq:PW:system}.

 \begin{lemma}\label{lem:3d:Hitchin}
 There exists an $\sunitary{2}$--equivariant, $G$--invariant map $\Lie{g}\otimes \R^3\ra \bigoplus_{i=1}^r{\tu{Sym}^{2d_i}(W)^\R}$, where $W$ is the standard representation of $\sunitary{2}$.
 \proof
 Think of $\Lie{g}\otimes \R^3$ as $H^0(\PP^1;\Lie{g}^\C\otimes\mathcal{O}(2))^\R$ and apply the map \[
 H^0(\PP^1;\Lie{g}^\C\otimes\mathcal{O}(2))^\R\longrightarrow \bigoplus_{i=1}^r{H^0(\PP^1;\Lie{g}^\C\otimes\mathcal{O}(2d_i))^\R}
 \]
 induced by a basis of Ad-invariant polynomials on $\Lie{g}$.
 \endproof	
 \end{lemma}
 
Now, $\phi\in\Omega^1(Q;\ad\, P)$ is a section of a bundle with fibre $\Lie{g}\otimes \R^3$. Fibrewise application of the map of Lemma \ref{lem:3d:Hitchin} yields a section $s_\phi$ of $\bigoplus_{i=1}^r{\tu{Sym}^{2d_i}(\slashed{S})^\R}$, where $\slashed{S}$ is the spinor bundle of $Q$.

 
In the analogy with the Calabi--Yau setting, the map $\phi\mapsto s_\phi$ plays the role of Hitchin's map since the fibres of $\bigoplus_{i=1}^r{\tu{Sym}^{2d_i}(\slashed{S})^\R}$ over $Q$ are precisely the parameter space of the hyperk\"ahler metrics of Theorem \ref{thm:ALE}. As in the special case $[\phi,\phi]=0$, we expect that \eqref{eq:PW:system} implies that the map $s_\phi$ satisfies Donaldson's adiabatic limit equations and therefore provides the data to define an approximate (generally incomplete!) $\gtwo$--metric deforming the model metric $g_{\mathcal{W}_0}$.   
     
\begin{remark*}
The incomplete hyperk\"ahler metrics of Theorem \ref{thm:ALE} do not arise in the Calabi-Yau setting of \cite{Diaconescu:Donagi:Pantev} since they are automatically excluded if we fix the moduli of one K\"ahler form in the hyperk\"ahler triple.	
\end{remark*}

\end{document}